\documentclass[11pt]{amsart}

\usepackage{fullpage}
\usepackage{xcolor}
\usepackage{graphicx}
\usepackage{epstopdf}
\usepackage{url}
\usepackage{latexsym}
\usepackage{amsfonts,amsmath,amssymb}

\newtheorem{theorem}{Theorem}[section]
\newtheorem{proposition}[theorem]{Proposition}
\newtheorem{lemma}[theorem]{Lemma}

{
\theoremstyle{definition}

\newtheorem{example}[theorem]{Example}
}

{
\theoremstyle{remark}

}

\numberwithin{equation}{section}

\date{\today}

\begin{document}

\title[One copy of a monotone pattern]{Permutations with exactly one copy of a monotone pattern of length $k$, and a generalization}

\author{Mikl\'os B\'ona}
\address{Department of Mathematics, University of Florida, 358 Little Hall, PO Box 118105,
Gainesville, FL 32611-8105}
\email{bona@ufl.edu}

\author{Alexander Burstein}
\address{Department of Mathematics, Howard University, Washington, DC 20059}
\email{aburstein@howard.edu}

\begin{abstract} We construct an injection from the set of permutations of length $n$ that contain exactly one copy of the decreasing pattern of length $k$ to the set of permutations of length $n+2$ that avoid that pattern. We then prove that the generating function counting the former is not rational, and in the case when $k$ is even and $k\geq 4$, it is not even algebraic. We extend our injection and our  nonrationality result to a larger class of patterns. 
\end{abstract}

\maketitle

\section{Introduction} \label{sec:intro}

We say that a permutation $p$ \emph{contains} the pattern $q=q_1q_2\cdots q_k$ 
if there is a $k$-element set of indices $i_1<i_2< \cdots <i_k$ so that $p_{i_r} < p_{i_s} $ if and only if $q_r<q_s$; in other words, if there is a $k$-element subsequence of $p$ that is order-isomorphic to $q$. If $p$ does not contain $q$, then we say that $p$ \emph{avoids} $q$. For example, $p=3752416$ contains $q=2413$, as the first, second, fourth, and seventh entries of $p$ form the subsequence 3726, which is order-isomorphic to $q=2413$.  For an introduction to permutation patterns, we refer the reader to a recent survey \cite{vatter} and a book \cite{combperm} on the subject.  

%\begin{notation} \label{not:snq}
In this paper, we will use $\mathfrak{S}_n(q)$ (resp. $\mathfrak{S}_n(q,1)$) to denote the set of all permutations of length $n$ that avoid $q$ (resp. contain \emph{exactly} one copy of $q$). Likewise, we will use $S_n(q)=|\mathfrak{S}_n(q)|$ and $S_n(q,1)=|\mathfrak{S}_n(q,1)|$ to denote the number of the corresponding permutations.
%\end{notation}

The numbers $S_n(q)$ have been the subject of vigorous research during the last thirty years, and while plenty of open questions remain,  some strong results have been found.  There are much fewer results concerning the numbers $S_n(q,1)$. 
%of permutations of length $n$ that contain \emph{exactly one} copy of the pattern $q$. 
For the pattern $q=321$,
John Noonan \cite{noonan} proved the exact formula 
\begin{equation} \label{fnoonan} S_n(321,1)=\frac{3}{n} {2n\choose n+3}.\end{equation} 
 More recently, the second author \cite{burstein} and Doron Zeilberger \cite{zeilberger} also
contributed to a deeper understanding of these permutations. 

It is well-known that $S_n(q)=\binom{2n}{n}/(n+1)$ for all patterns $q$ of length three. So formula \eqref{fnoonan} shows that $\lim_{n\rightarrow \infty} S_n(321,1) / S_n(321) =3$,  and therefore the number of $n$-permutations that contain one copy of $321$ is only a constant times more than the number
of $n$-permutations that contain no copies of 321. The analogous statement is not true for all patterns. For instance, it is known \cite{bona-one} that $S_n(132)={2n-3 \choose n-3}$, and so 
$\lim_{n\rightarrow \infty} S_n(132,1) / S_n(132) = \infty$. 

In this paper, we generalize some of the work on $q=321$ mentioned earlier  by proving that if $k\geq 3$, then  $\lim_{n\rightarrow \infty} S_n(k\cdots 21,1) / S_n(k\cdots 21) \leq C_k$, for a finite constant $C_k$. We then use this fact, and a recent result of the first author to show that the ordinary generating function $s_{k,1}(z)$ of the sequence   $\{S_n(k\cdots 21,1)\}_{n\ge 0}$ is not rational.  Information on the growth rate of the sequence $\{S_n(k\cdots 21,1)\}_{n\ge 0}$ enables us to prove that if $k>2$ is even, then $s_{k,1}(z)$ is not algebraic. 

A crucial tool in our proof is an injection from the set of all  permutations of length $n$ that contain exactly one
copy of the decreasing pattern of length $k$ to the set of permutations of length $n+2$ that avoid that
pattern. In Section \ref{sec:rhoisany}, we generalize this injection for a family of patterns, which enables us to generalize our nonrationality result as well. 

\section{An injection for the case of $k=3$} \label{kisthree}

Let $\mathfrak{S}_n(q)$ denote the set of all permutations of length $n$ that avoid $q$, and let 
$\mathfrak{S}_n(q,1)$ denote the set of all permutations of length $n$ that contain exactly one copy of $q$.
In this section, we use the analysis of $\mathfrak{S}_n(321)$ due to the second author \cite{burstein} and Zeilberger
\cite{zeilberger} to construct an injection $f:\mathfrak{S}_n(321,1) \to \mathfrak{S}_{n+2} (231) $. 
The important point is not simply constructing an injection, but to construct an injection that can be generalized
for longer patterns, as we will do in the next section. 

If $p$  is a permutation, then a \emph{right-to-left minimum} in $p$ is an entry that is smaller than all entries to its right. For instance, if $p=316254$, then the right-to-left minima of $p$ are the entries 1, 2, and 4.  

We will need the following simple fact, which is a direct consequence of the classic Simion-Schmidt bijection \cite{simion} after taking reverses. Let $p=p_1p_2\cdots p_n$ be a $321$-avoiding permutation. Keep the positions and values of the right-to-left minima of $p$ fixed, and fill the remaining slots with the remaining entries from right to left so that each slot is filled with the smallest entry that can be inserted into that slot without creating a new right-to-left minimum. For instance, if $p=3\underline{1}46\underline{2}7\underline{5}$, with the right-to-left minima 1, 2, and 5 at positions 2, 3, and 6, respectively, then $g(p)=7\underline{1}43\underline{2}6\underline{5}$.

\begin{proposition} \label{simion-sch}
The map $g$ defined above is a bijection from the set $\mathfrak{S}_n(321)$ to the set $\mathfrak{S}_n(231)$. 
\end{proposition}

\begin{proof}
Note that $g(p)\in \mathfrak{S}_n(231)$, since if $g(p)$ contained a copy of 231, it would contain one that ends in a right-to-left minimum $r$, but that is impossible because that 
would mean that two entries that are larger than $r$ were not placed in the order specified by the algorithm defining $g$. Furthermore, $g$ is a bijection, because given 
$p^*\in \mathfrak{S}_n(231)$, we can recover the unique preimage of $p^*$ under $g$ by keeping the positions and values of the right-to-left minima of $p^*$ fixed and writing the remaining entries in the remaining slots in increasing order. 
\end{proof}

Let $\pi\in \mathfrak{S}_n(321,1)$. Let $cba$ be the only copy of 321 in $\pi$, and  
let 
\[\pi = \pi_1 \ c \ \pi_2 \ b \ \pi_3 \ a \  \pi_4 ,\] where the $\pi_r$ are blocks of entries of $\pi$. 

Then the only entry to the left of $b$ that is larger than $b$ is $c$, and the only entry to the right of $b$ that is smaller than $b$ is $a$, otherwise $cba$ would not be the only 321-copy in $\pi$. For the same reason, all entries in $\pi_2$ are smaller than $a$, and all entries in $\pi_3$ are greater than $c$. Therefore, the permutation $\sigma_1=\pi_1 b \pi_2 a$ is a 321-avoiding permutation on the set $\{1,2,\,\cdots ,b\}$, and
the permutation $\sigma_2 = c \pi_3 b \pi_4$  is a 321-avoiding permutation on the set $\{b,b+1,\,\cdots,n\}$. Let us add 1 to each entry of $\sigma_2$, yielding the permutation
$\sigma_2'$, which is a 321-avoiding permutation on the set $\{b+1,b+2,\,\cdots,n+1\}$. 

Finally, we define $f(\pi)$ as the concatenation of $g(\sigma_1)$, the entry $n+2$, and $g(\sigma_2')$, where $g:\mathfrak{S}_n(321) \rightarrow \mathfrak{S}_n(231)$ is the
reverse Simion-Schmidt bijection whose existence we proved in Proposition \ref{simion-sch}. 

\begin{example} \label{expi}
Let $\pi = 25147386$. Then $cba=543$, $(\pi_1,\pi_2,\pi_3,\pi_4)=(2,1,7,86)$, so $\sigma_1 =2413 $, and $\sigma_2=57486$, so $\sigma_2'=68597$. Therefore, $g(\sigma_1)=4213$ and $g(\sigma_2') =  96587$, so 
\[f(\pi) = 4\ 2 \ 1 \ 3 \  10 \ 9 \ 6 \ 5  \ 8 \ 7 .\]
\end{example}

\begin{proposition} \label{prop:321-1-231}
The map $f:\mathfrak{S}_n(321,1) \to \mathfrak{S}_{n+2} (231) $ is an injection. 
\end{proposition}

\begin{proof} It is clear that $f(\pi) \in  \mathfrak{S}_{n+2} (231) $.  We recover $g(\sigma_1)$ and $g(\sigma_2')$ as the subsequences of entries to the left of $n+2$ and to the right of $n+2$, respectively. The result follows since $g$ is injective. 
\end{proof}

It is the following property of the map $f$ that will make it useful for us in the next section. Let $\pi_3'$ (resp. $\pi_4'$) denote the image of $\pi_3$ (resp. $\pi_4$) in $\sigma_2'$, that is, we obtain  $\pi_3'$ (resp. $\pi_4'$) by adding 1 to each entry of $\pi_3$ (resp. $\pi_4$). 

\begin{proposition} \label{rilmin}
The right-to-left minima of $f(\pi)$ are, from right to left, the set of right-to-left minima in $\pi_4'$ with their positions shifted right by $2$, the entry $b+1$ in position $\pi^{-1}(a)+2$, the entry $a$ in position $\pi^{-1}(b)$, and the subset of right-to-left minima of $\pi$ that are located in $\pi_1$ or $\pi_2$, with their positions unchanged.
\end{proposition}

\begin{proof} As $g$ keeps the right-to-left minima fixed, it suffices to locate the right-to-left minima of the permutation $\sigma_1 (n+2) \sigma_2'$. Then our statement follows from the
fact that all entries in $\pi_4'$ and $\pi_3'$ are larger than $b+1$, and that $a<b+1$. 
\end{proof}

\begin{example} Consider the permutation $\pi$ of Example \ref{expi}. The right-to-left mimina of $\pi$ are $6$, $a=3$, and $1$, at positions $8$, $\pi^{-1}(a)=6$, and $3$. Proposition \ref{rilmin} states that the right-to-left minima of $f(\pi)$ are, from right to left, $7=6+1$, $b+1=5$, $a=3$, and $1$, in positions $10=8+2$, $8=6+2=\pi^{-1}(a)+2$, $4=\pi^{-1}(b)$, and $3$, respectively.  We can easily verify that these 
are indeed the right-to-left minima of $f(\pi)$. 
\end{example}

\section{The case of general $k$} \label{sec:kisany}

Let $p=p_1p_2\cdots p_n$ be a permutation. Define the \emph{co-rank} of an entry $p_i$ as the length of the longest decreasing subsequence starting at $p_i$. For example, if $p=361254$, then entries 1, 2, and 4 have co-rank 1, entries 3 and 5 have co-rank 2, and entry 6 has co-rank 3.  In general, right-to-left minima are exactly the entries of co-rank 1. Entries of the same co-rank form an increasing subsequence in $p$. If $p\in \mathfrak{S}_n(k\cdots 21)$, then the co-rank of each entry in $p$ is at most $k-1$, while if $p\in \mathfrak{S}_n (k\cdots 21,1)$, then exactly one entry in $p$ has co-rank $k$.  

In what follows, we will use a geometric representation of a permutation $p=p_1 p_2 \cdots p_n$ as a \emph{permutation diagram}, i.e. a set of points $\{(i,p_i)\mid 1\le i\le n\}$, where we think of the first coordinate as increasing from west to east, and the second coordinate as increasing from south to north. In particular, we say that the entry $p_i$ is \emph{northwest} of the entry $p_j$ if $i<j$ and $p_i>p_j$.

\begin{theorem} \label{maininj}
Let $k\geq 3$. Then there exists an injection \[F_k: \mathfrak{S}_n(k\cdots 21,1) \rightarrow \mathfrak{S}_{n+2} ((k-1)k\cdots 21) .\]
\end{theorem}

\begin{proof}
For $k=3$, the map $F_3$ is the map $f$ that we defined in Proposition \ref{prop:321-1-231}. Now let us assume that $k>3$. Let $p\in  \mathfrak{S}_n(k\cdots 21,1)$, and let $\pi$ be the subsequence of entries in $p$ of co-rank $k$, $k-1$, or $k-2$. Then $\pi \in  S_m(321,1) $ for some $m<n$. Call the entries of $\pi$ \emph{blue}.  Let $\tau$ denote the rest of $p$, and call the entries in $\tau$ \emph{red}. Notice that no red entry in $p$ is to the northwest of any blue entry. 

We define $F_k(p)$ by replacing $\pi$ with $f(\pi)$ in $p$ in two steps. First, we replace the subsequence $\pi = \pi_1 \ c \ \pi_2 \ b \ \pi_3 \ a \  \pi_4$ of $p$ with the subsequence $h(\pi) = \pi_1 \ b \ \pi_2 \ a \ n+2  \ c+1 \  \pi_3' \ b+1 \  \pi_4'$ and adjust $\tau$ as follows:
\begin{itemize}
\item replace the entries $c$, $b$, and $a$, respectively, with the entry $b$, block $a\, (n+2)\,(c+1)$, and the entry $b+1$, respectively, and color the new entries, except for $n+2$, blue;
\item add 1 to every entry in $\pi_3$ and $\pi_4$ (to obtain $\pi_3'$ and $\pi_4'$) and color the new entries blue;
\item add 1 to every entry of $\tau$ greater than $b$ and color the new entries red.
\end{itemize}

Then we replace the subsequence $h(\pi)$ with $f(\pi)$ by applying the map $g$ to the subsequences $\pi_1\, b\, \pi_2\, a$ and $(c+1)\, \pi_3'\, (b+1)\, \pi_4'$ of $h(\pi)$. This yields $F_k(p)$. Note that this step preserves the positions and values of the right-to-left minima of $h(\pi)$, since there are no right-to-left minima in $h(\pi)$ between the right-to-left minima $a$ and $b+1$.

Observe that after this operation there is still no red entry northwest of a blue entry. To show this, it suffices to prove that no red entry is northwest of a blue entry that is a right-to-left minimum of $f(\pi)$, or equivalently, of $h(\pi)$. This is clearly true for all the right-to-left minima in $h(\pi)$ other than $a$ and $b+1$.

If a red entry $x\in F_k(p)$ is to the northwest of $b+1$ in $F_k(p)$, then $x$ would be to the northwest of $a$ in $p$, a contradiction. If a red entry $y$ is to the northwest of $a$ in $F_k(p)$, then $y$ is to the northwest of $a$ in $p$, which is again a contradiction. It is similarly easy to verify that the new entries all have rank $k-2$ or higher.

As $f(\pi)\in \mathfrak{S}(231)$, it follows that $F_k(p) \in  S_{n+2} ((k-1)k\cdots 21)$. Indeed, if $F_k(p)$ contained a copy of $(k-1)k\cdots 21$, then it would contain one whose
last entry has co-rank at least 1, the one before that has co-rank at least 2, and so on, continuing to the fourth entry from the left, which would have co-rank at least $k-3$. The remaining three entries would all have to have a higher co-rank and form a $231$-pattern, but that is impossible since entries of co-rank at least $k-2$ are in $f(\pi)$ and form a $231$-avoiding permutation. 

So, given $F_k(p)$, we can determine the set of red entries, and the set of blue entries in $F_k(p)$, and therefore, in $p$.
This implies that if $F_k(p)=F_k(p^*)$, then $\pi =\pi^*$ and $\tau =\tau^*$ must hold, since the action of $f$ is injective. So $F_k$ is an injection as claimed. 
\end{proof}

\begin{example}
Let $k=4$, and let $p=481593276\in \mathfrak{S}_9(4321,1)$. Then $\tau=126$ and $\pi=4\mathbf{8}\mathbf{5}9\mathbf{3}7$ (where the unique occurrence of the pattern $321$ in $\pi$ is marked in bold), so that $p=\textcolor{blue}{48}\textcolor{red}{1}\textcolor{blue}{593}\textcolor{red}{2}\textcolor{blue}{7}\textcolor{red}{6}$ and $\pi_1=4$, $c=8$, $\pi_2$ is empty, $b=5$, $\pi_3=9$, $a=3$, and $\pi_4=7$. Moreover, in $\tau$, the entries $1$ and $2$ are less than $b$, whereas $6$ is greater than $b$. Therefore, we first replace $\pi$ with $h(\pi)$ as in the proof of Theorem~\ref{maininj} and add 1 to the entries of $\tau$ greater than $b$ to obtain the string
\[
\textcolor{blue}{4\ 5}\ \textcolor{red}{1}\ \textcolor{blue}{3}\ 11\ \textcolor{blue}{9\ 10\  6}\ \textcolor{red}{2}\ \textcolor{blue}{8}\ \textcolor{red}{7},
\]
then replace the subsequences $\pi_1 b \pi_2 a = 453$ and $(c+1)\pi_3'(b+1)\pi_4'=9(10)68$ with $g(\pi_1 b \pi_2 a)=g(453)=543$ and $g((c+1)\pi_3'(b+1)\pi_4')=g(9(10)68)=(10)968$, respectively, so that $f(\pi)=543(11)(10)968$ and
\[
F_4(p)= \textcolor{blue}{5\ 4}\ \textcolor{red}{1}\ \textcolor{blue}{3}\ 11\ \textcolor{blue}{10\ 9\ 6}\ \textcolor{red}{2}\ \textcolor{blue}{8}\ \textcolor{red}{7} \in \mathfrak{S}_{11}(3421).
\]
%
%
% and as the right-to-left minima
%of this permutation are, from the right, 9, 6, 4, 3, and 2, we get that $f(\pi)= 28346(11)(10)9$. Thus, 
%\[
%F_4(p) =   2\ 8 \ 3\ 4\ 1\ 6\ 11 \ 10 \ 9\  5\ 7 \in \mathfrak{S}_{11} (3421).
%\]
\end{example}

%================ begin Alex =================

\section{The general-$k$ case generalized} \label{sec:rhoisany}

In this section, we will further generalize the case of the anti-identity pattern $k\cdots21$ to a whole class of patterns. For this we will need to define the notions of direct sums, skew-sums and domination. 
%%need a reference here: possibly albert, atkinson et al.

%Given a string $\sigma$ define its \emph{$m$-lift} $\sigma^{(m)}$ by $\sigma^{(m)}(i)=\sigma(i)+m$, where $1\le i\le |\sigma|$. 

Given two permutations $\sigma_1$ and $\sigma_2$, define their \emph{direct sum} $\sigma_1\oplus\sigma_2$ and \emph{skew-sum} $\sigma_1\ominus\sigma_2$ as permutations of length $|\sigma_1|+|\sigma_2|$ such that
\[
\begin{split}
(\sigma_1\oplus\sigma_2)(i)&=
\begin{cases}
\sigma_1(i), & i\le |\sigma_1|,\\
\sigma_2(i-|\sigma_1|)+|\sigma_1|, & i>|\sigma_1|,
\end{cases}\\
(\sigma_1\ominus\sigma_2)(i)&=
\begin{cases}
\sigma_1(i)+|\sigma_2|, & i\le |\sigma_1|,\\
\sigma_2(i-|\sigma_1|), & i>|\sigma_1|.
\end{cases}
\end{split}
\]
For example, $231=12\ominus 1=(1\oplus 1)\ominus 1$ and $321=21\ominus 1 = 1\ominus 1\ominus 1$.

Furthermore, 
%for a permutation $\sigma$, let the \emph{permutation diagram} (or simply \emph{diagram}) of $\sigma$ be the set $D(\sigma)=\{(i,\sigma(i))\mid 1\le i\le |\sigma|\}$. 
we say that a point $(i_1,j_1)$ \emph{dominates} a point $(i_2,j_2)$ if $i_1<i_2$ and $j_1>j_2$, or equivalently, if $(i_1,j_1)$ is northwest of $(i_2,j_2)$. More generally, for a point $P$ and a set of points $S$, we say that $P$ dominates $S$ if $P$ dominates every point in $S$.

Let $k\ge 3$ and let $\rho=(k-3)\cdots 21$, and consider the construction in the previous section. Clearly, $p=k\cdots 21=321\ominus\rho$, and the points of co-rank $k$ are those in the permutation diagram of $\sigma$ that dominate an occurrence of pattern $21\ominus\rho$, while the points of co-rank $k-1$ are those in the permutation diagram of $\sigma$ that dominate an occurrence of $1\ominus\rho$ but not of $21\ominus\rho$, and finally, the points of co-rank $k-2$ are those in the permutation diagram of $\sigma$ that dominate an occurrence of $\rho$ but not of $1\ominus\rho$. This suggests the following generalization of Theorem \ref{maininj}.

\begin{theorem} \label{thm:geninj}
Let $k\geq 3$ and let $\rho\in \mathfrak{S}_{k-3}$. Then there exists an injection 
\[
F_k: \mathfrak{S}_n(321\ominus\rho,1) \to \mathfrak{S}_{n+2} (231\ominus\rho).
\]
\end{theorem}

\begin{proof}
The proof of this theorem parallels that of Theorem \ref{maininj}. The construction in Section \ref{sec:kisany} can be generalized as follows. Let $\rho$ be \emph{any} pattern, and consider $p=321\ominus\rho$. Call any entry of $p$ that dominates an occurrence of $\rho$ \emph{blue}, and call the rest of the entries \emph{red}. As in the Section \ref{kisthree}, let $\pi$ be the subsequence of the blue entries of $p$, and let $\tau$ be the subsequence of the red entries of $p$.

Given a permutation $p\in \mathfrak{S}_n(321\ominus\rho,1)$, we define $F_k(p)$ similarly by replacing the subsequence $\pi$ of the blue entries of $p$ with $f(\pi)$ and adding 1 to the entries of $\tau$ greater than $b$. Then, as before, no red entry is to the northwest of any blue entry, and the rest of the proof follows exactly as in the proof of Theorem~\ref{maininj}.
\end{proof}

%================= end Alex ==================

\section{Nonrationality} \label{sec:nonrat}

\subsection{Monotone patterns}
In what follows, we will use the following well-known fact that interchanging the first
two entries of a decreasing pattern does not change the number of permutations avoiding that pattern. For the proof of this fact, see Exercise 1 in Chapter 4 of \cite{combperm}. 

\begin{proposition} \label{equality}
For all $k\geq 3$, the equality \[S_n(k\cdots 21)=S_n((k-1)k\cdots 21) \]
holds.
\end{proposition}

If two patterns $q$ and $q'$ satisfy the equality $S_n(q)=S_n(q')$ for all $n$, then we will call them \emph{Wilf-equivalent}. 
We say that a pattern $q$ is \emph{skew-indecomposable} if there do not exist nonempty patterns $\alpha$ and $\beta$ such that $q=\alpha \ominus \beta$; otherwise, we say that $q$ is \emph{skew-decomposable}. If $q$ is skew-decomposable, then there is a unique way to decompose it into maximal blocks of consecutive entries so that each such 
block is skew-indecomposable. We call these blocks the \emph{skew-blocks} of $q$. For instance, $q=35412$ has two skew-blocks, $354$ and $12$. 

\begin{lemma} \label{value}
Let $q=q_1q_2\cdots q_k$ be a pattern such that either $\{1,k\}\neq \{q_1,q_k\}$, or $q$ is Wilf-equivalent to a pattern $v=v_1v_2\cdots v_k$ so that $\{1,k\}\neq \{v_1,v_k\}$. Let $r$ be the convergence radius of the ordinary generating function $A(z)=\sum_{n\geq 0} S_n(q)z^n$ of the sequence $S_n(q)$. Then 
\[
A(r) < \infty.
\]
\end{lemma}

\begin{proof}
We can assume that $q$ is \emph{skew-indecomposable}.  Indeed, if $q$ is skew-decomposable, we can replace $q$ by its reverse, and that will be
skew-indecomposable, while the other conditions will not be affected. 

Let $S_{n,j}(q)$ be the number of $q$-avoiding permutations of length $n$ that consist of $j$ skew-blocks. It is shown in \cite[Lemma 4.2]{bona} that for all $n$ and all $q$ satisfying the conditions of this lemma, the inequality 
\begin{equation} \label{onevstwo} 
S_{n,2}(q) \leq S_{n,1}(q) 
\end{equation} 
holds. 

Let $A_1(z)=\sum_{n\geq 1} S_{n,1}(q)z^n$, and let $r_1$ be the radius of convergence of $A_1$. Similarly, let $A_2(z)=\sum_{n\geq 2} S_{n,2}(q)z^n$, and note that $A_2(z)=A_1(z)^2$. 

We claim that $A_1(r_1) < 1$. Indeed, let us first assume that $A_1(r_1) > 1$. As $A_1(z)$ has nonnegative real coefficients, it is a monotone increasing function on the interval $(0,r_1)$.  So, there exists a real number $x\in (0,r_1)$ such that $1 < A_1(x) < A_1(r_1)$. As real numbers greater than $1$ increase when squared, this implies
\begin{equation} \label{eq:1vs2ogf}
\sum_{n \geq 1} S_{n,1}(q) x^n =A_1(x) < A_1(x)^2=A_2(x)=\sum_{n \geq 2} S_{n,2}(q) x^n,
\end{equation}
which clearly contradicts (\ref{onevstwo}). If $A_1(r_1)=1$, there is still a contradiction, because the series $A_1(r_1)$ has a positive summand $S_{1,1}(r_1)$ that has no match in the series $A_2(r_1)$, and otherwise each summand of $A_1(r_1)$ is at least as large as the summand of $A_2(r_1)$ with the same power of $r_1$. This implies that
\[
1=A_1(r_1)\ge S_{1,1}(r_1)+A_2(r_1)>A_2(r_1)=A_1(r_1)^2=1,
\] 
which is impossible. Thus, $A_1(r_1) < 1$.

It is proved in \cite[Theorem 6.2]{bona} that $A(z)$ and $A_1(z)$ have the same convergence radius, that is, with the notation of this paper, $r=r_1$. 
Therefore, 
\[
A(r)=\frac{1}{1-A_1(r)} < \infty,
\]
since we have seen that $A_1(r)=A_1(r_1)<1$. 
\end{proof}

Now we are in a position to state and prove the main result of this section.  Recall that $S_n(q,1)$ denotes the number of permutations of length $n$ that contain exactly 
one copy of the pattern $q$. This is not to be confused with the number $S_{n,i}(q)$ of $q$-avoiding permutations of length $n$ that consist of $i$ skew block that we
discussed in the proof of Lemma \ref{value}.

\begin{theorem} \label{nonrat} 
Let  $k\geq 3$, and let $s_{k\cdots 21,1}(z)=\sum_{n\geq 0} S_n(k\cdots 21,1)z^n$, the ordinary generating function of the sequence $S_n(k\cdots 21,1)$. Then $s_{k\cdots 21,1}(z)$ is not rational. 
\end{theorem}

\begin{proof} 
Theorem \ref{maininj} shows that $S_n(k\cdots 21,1)\leq S_{n+2}((k-1)k\cdots 21)$, while it is straightforward to see that $S_n(k\cdots 21,1)\geq S_{n-k}(k\cdots 21)=S_{n-k}((k-1)k\cdots 21)$, where the equality follows from Proposition \ref{equality}, and the second inequality follows from the observation that if $p\in \mathfrak{S}_{n-k}(k\cdots 21)$ then $p\oplus k\cdots 21\in \mathfrak{S}_n(k\cdots 21,1)$. These inequalities imply
that the exponential order of the sequence $S_n(k\cdots 21,1)$ is the same as that of the sequence
$S_n(k\cdots 21)$, and the latter is well-known \cite{regev} to be $(k-1)^{2n}$. 

Therefore, the generating functions $s_{k\cdots 21,1}(z)$ and $s_{k\cdots 21}(z)$ of these sequences have the same convergence radius $R=1/(k-1)^2$. 

As the decreasing pattern $k(k-1)\cdots 21$ satisfies the conditions of Lemma \ref{value}, it follows that
\begin{equation} \label{finite}
s_{k\cdots 21}(R) < \infty .
\end{equation}

This implies that
\[s_{k\cdots 21,1}(R) \leq \sum_{n\geq 0} S_{n+2}(k\cdots 21) R^n =\frac{s_{k\cdots 21}(R) -R -1}{R^2} < \infty,\]
The first inequality holds because each coefficient of $s_{k\cdots 21,1}(z)$ is at most as large as the corresponding coefficient of the generating
function $ \sum_{n\geq 0} S_{n+2}(k\cdots 21) R^n$.  

If $s_{k\cdots 21,1}(z)$ were a rational function, then all its singularies would be poles. As all coefficients of $s_{k\cdots 21,1}(z)$ are nonnegative real numbers, we know by Pringsheim's Theorem \cite[Theorem IV.6]{Flajolet} that $R$ itself  is a singularity of $s_{k\cdots 21,1}(z)$. However, that would imply that $R$ is a pole, and so
$s_{k\cdots 21,1}(R)=\infty$. Therefore, $s_{k\cdots 21,1}(z)$ is not a rational function. 
\end{proof} 

\subsection{A more general result}
In order to generalize Theorem \ref{nonrat} for the class of patterns discussed in Section \ref{sec:rhoisany}, we need the following generalization of 
Proposition \ref{equality}. 

\begin{theorem}[\cite{babson}] \label{babson}
Let $\tau\in \mathfrak{S}_{n-2}$ be any pattern. Then for all nonnegative integers $n$, the equality $S_n(21 \ominus \tau) = S_n(12\ominus \tau)$
holds.
\end{theorem}

Now we can state and prove our generalization of Theorem \ref{nonrat}. 

\begin{theorem} Let $k\geq 3$ and $\rho\in \mathfrak{S}_{k-3}$, and let $A_{321 \ominus \rho,1}(z)$ be the ordinary generating function of the sequence $S_n(321 \ominus \rho,1)$.
Then $A_{321 \ominus \rho,1}(z)$ is not rational.
\end{theorem} 

\begin{proof}
Theorem \ref{thm:geninj} shows that 
$S_n(321 \ominus \rho,1)\leq S_{n+2}(231 \ominus \rho)$. On the other hand, $S_n(321 \ominus \rho,1)\geq S_{n-k}(321 \ominus \rho)=S_{n-k}(231 \ominus \rho)$, where the last equality follows from Theorem \ref{babson}. The inequality, as in Theorem \ref{nonrat}, follows from the observation that if $p\in \mathfrak S_{n-k}$, then $p\oplus (321 \ominus \rho)\in \mathfrak S_n(321 \ominus \rho,1)$. Therefore, the sequences $S_n(321 \ominus \rho,1)$ and  $S_{n+2}(231 \ominus \rho)$ have the same exponential order, and their generating functions have the same convergence radius $R^*$.

 The pattern $231 \ominus \rho$ does not start
with its largest or smallest entry, so it satisfies the conditions of Lemma \ref{value}, and therefore,  if 
$A_{231 \ominus \rho}(z)$ is the generating function for the numbers $S_n(231 \ominus \rho)$, then
\[A_{231 \ominus \rho}(R^*) < \infty . \]
Finally, we can compare the values of $A_{231 \ominus \rho}(R^*)$ and 
 $A_{321 \ominus \rho,1}(R^*)$  as we did in the proof of Theorem \ref{nonrat}, and conclude that $A_{321 \ominus \rho,1}(R^*)<\infty$ .
\end{proof}
 
\section{Nonalgebraicity for monotone patterns}
In this section we return to the context of monotone patterns.  We show that if $k>2$ is even, then the generating function
$s_{k\cdots 21,1}(z)$ is not algebraic, because its coefficients grow
in a way in which the coefficients of no algebraic power series do.

We need the following classic result of Amitaj Regev.
\begin{theorem}[\cite{regev}] \label{regev}
For all $k\geq 2$, there exists a constant $\gamma_k$ so that the asymptotic equality
\[S_n(k\cdots 21) \simeq \gamma_k \frac{(k-1)^{2n}}{n^{(k^2-2k)/2} }\]
holds.
\end{theorem}  

We have seen in the proof of Theorem \ref{nonrat} that 
\[S_{n-k}(k\cdots 21) \leq S_n(k\cdots 21,1) \leq S_n((k-1)k\cdots 21) =S_{n+2}(k\cdots 21).\]
So Theorem \ref{regev} implies that there exist constants $c$ and $C$ so 
that 
\begin{equation} \label{chainin} c\frac{(k-1)^{2n}}{n^{(k^2 -2k)/2}} \leq  S_n(k\cdots 21,1) \leq C \frac{(k-1)^{2n}}{n^{(k^2 -2k)/2}}.\end{equation}

Now we can state and prove the main result of this section. 
\begin{theorem} \label{nonalg}
Let $k>2$ be an \emph{even} integer. Then the generating function $s_{k\cdots 21,1}(z)$ is not algebraic.
\end{theorem}

Note that if $k>2$ is even, then the exponent $(k^2-2k)/2$ of $n$ in the denominator above is a positive integer that is larger than $1$. This fact makes the following the key lemma applicable in the proof of Theorem \ref{nonalg}. We are indebted to Alin Bostan for the rigorous proof of this lemma. %\cite{bostan}. 

\begin{lemma} \label{bostan}
 Let $f(z)=\sum_{n\geq 0}f_nz^n$ be a power series with nonnegative real coefficients that is analytic at the origin. Let us assume that constants $c$, $C$, $K$ and $m$ exist so that $m>1$ is an integer, and for all positive integers $n$, the chain of inequalities
\begin{equation} \label{condition}  
c \frac{K^n}{n^m}  \leq f_n  \leq C \frac{K^n}{n^m}
\end{equation} 
holds. Then $f(z)$ is not an algebraic power series.
\end{lemma}

\begin{proof}
Let us replace $z$ by $z/K$. This transformation preserves algebraicity, and it turns condition (\ref{condition}) into the chain of inequalities
\begin{equation} \label{newcondition}   \frac{c}{n^m}  \leq g_n  \leq  \frac{C}{n^m}.\end{equation}
%Note that $\lim_{n\rightarrow \infty} g_n =0$. 

Let us assume by contradiction that $f(z)$ is algebraic, or equivalently, that $g(z)= \sum_n g_n z^n$ is algebraic. By \cite[Theorem D]{Flajolet}, there exists a positive algebraic number $\beta$ and a rational number $s$ that is \emph{not a negative integer} so that
\[
g_n \sim  \frac{\beta^n \cdot  n^s}{\Gamma(s+1)}  \sum_{j=0}^r C_j \omega_j^n,
\]
with  the $C_j$ and the $\omega_j$  being algebraic constants, with $|\omega_j|=1$ for all $j$. Set $h_n=g_n \cdot n^m$, then
\begin{equation} \label{newestcondition}   
c \leq h_n  \leq  C
\end{equation}
and
\begin{equation} \label{hasym} 
h_n \sim  \frac{\beta^{n+m} \cdot  n^{s+m}}{\Gamma(s+1)}  \sum_{j=0}^r C_j \omega_j^n.
\end{equation}

Equation (\ref{newestcondition}) shows that the sequence $\{h_n\}_{n\ge 0}$ is bounded, while the finite sum  $\sum_{j=0}^r C_j \omega_j^n$ is bounded
since $|\omega_j|=1$ for all $j$. Therefore, Equation \eqref{hasym} implies that 
\[
A \leq \beta^{n+m} n^{s+m} \le B 
\]
for some positive constants $A$ and $B$. However, this can only happen if $\beta =1$ and $s=-m$, contradicting the assumption that $s$ is not a negative integer. 
\end{proof}

The proof of Theorem \ref{nonalg} is now immediate.
\begin{proof}[Proof of Theorem \ref{nonalg}]
If $k>2$ is even, then the exponent $(k^2 -2k)/2$ of $n$ in the denominator of both bounds in Equation \eqref{chainin} is an integer greater than 1. Therefore, by Lemma \ref{bostan}, the generating function $s_{k\cdots 21,1}(z)$ of the sequence  $S_n(k\cdots 21,1) $ is not algebraic.
\end{proof}

Note that the proof of this theorem was made possible by the fact that Theorem \ref{regev} describes the values of $S_n(k(k-1)\cdots 21)$ with a high level of precision, which 
enables us to use Lemma \ref{bostan}. There are no known results for other infinite families of patterns that would provide the growth rate of similar permutation classes at polynomial precision, therefore we were not able to prove non-algebraicity in a more general context. 

\medskip 

\begin{center} 
\textbf{Acknowledgment} 
\end{center}

We are indebted to Alin Bostan for a rigorous proof of Lemma~\ref{bostan} (personal communication, January 8, 2021). We are also grateful to Jay Pantone and Tony Guttmann for helpful remarks. We thank our anonymous referees for a careful reading
of the manuscript.

\end{document}